\documentclass[11pt]{article}
\usepackage[margin=1in]{geometry}
\usepackage{amsmath,amssymb,amsthm,mathtools}
\usepackage{booktabs}
\usepackage{array}

\newtheorem{theorem}{Theorem}[section]
\newtheorem{lemma}[theorem]{Lemma}

\newtheorem{conjecture}[theorem]{Conjecture}
\theoremstyle{definition}
\newtheorem{definition}[theorem]{Definition}
\theoremstyle{remark}
\newtheorem{remark}[theorem]{Remark}

\newcommand{\tr}{\operatorname{tr}}
\newcommand{\wt}{\operatorname{wt}}
\newcommand{\Z}{\mathbb Z}
\newcommand{\N}{\mathbb N}

\title{Proof of the Tu--Deng Conjecture}
\author{Thomas W. Cusick$^a$ \footnote{email: cusick@buffalo.edu~ORCID 0000-0002-0087-8855}
\vspace{.5cm}\\
$^a$\small Department of Mathematics,
\small University at Buffalo\\
\small 244 Mathematics Bldg.,  Buffalo, NY 14260\\}

\begin{document}
\maketitle

\begin{abstract}
We give a complete proof of the 2011 Tu--Deng conjecture.  We begin from its
original modular pair-count formulation, prove an equivalent cyclic
Hamming weight-drop formulation, and establish the exact transfer identity that
connects this count with a two-variable matrix polynomial.  The proof then
reduces the conjecture to normalized inequalities for the coefficients of that
polynomial. A 2011 conjecture by the author which came to be called the Cusick 
Conjecture (it is a consequence of the Tu--Deng Conjecture) was proved by K. Cheng
in 2026.  The proof in the present paper extends the cyclic deletion ideas of Cheng.
The new ideas might be applicable to other problems.
\end{abstract}

\section{The Tu--Deng conjecture and its history}\label{sec:introduction}
For a nonnegative integer $n$, let $\wt(n)$ denote the Hamming weight of its
binary expansion, that is, the number of digits equal to $1$.  Fix $k\ge2$ and
define $M=2^k-1$.  For $1\le t<M$, define
\[
 S_{t,k}=\bigl\{(a,b)\in\{0,1,\ldots,M-1\}^2:
 a+b\equiv t\pmod M,\ \wt(a)+\wt(b)<k\bigr\}.
\]
The Tu--Deng conjecture is the following finite assertion.

\begin{conjecture}[Tu--Deng]\label{TD}
For every $k\ge2$ and every $1\le t<2^k-1$,
\[
 |S_{t,k}|\le 2^{k-1}.
\]
\end{conjecture}

Tu and Deng introduced this conjecture in their study of cryptographic Boolean
functions \cite{TuDeng}.  Algebraic immunity measures the least degree of a
nonzero Boolean function annihilating a given function or its complement and
is a standard measure of resistance to algebraic attacks.  Assuming the
conjecture, Tu and Deng obtained two classes of Boolean functions with optimal
algebraic immunity: one class consists of bent functions, while the other
consists of balanced functions with optimal algebraic degree and very high
nonlinearity.  They verified the conjecture computationally for $k\le29$
\cite{CusickLiStanica, TuDeng}.

Cusick, Li, and St\u anic\u a studied the counting problem directly
\cite{CusickLiStanica}.  Using three different
counting approaches, they proved the conjecture for many families of
parameters, including several cases controlled by the binary weight of $t$ or
of its complement.  Their analysis also emphasized a central difficulty: the
relevant pairs have a scattered distribution, so a direct exact enumeration
appears difficult in general. While working on \cite{CusickLiStanica}, the
author formulated the sum-of-digits conjecture stated in the next paragraph,
which came to be called the Cusick Conjecture.

A major asymptotic advance was obtained by Spiegelhofer and Wallner
\cite{SpiegelhoferWallner}.  They proved that the proportion of
$t\in\{1,\ldots,2^k-2\}$ for which the Tu--Deng Conjecture holds tends to $1$ as
$k\to\infty$.  They also showed that the Tu--Deng conjecture implies Cusick's
sum-of-digits conjecture (see \cite{DKS}), which asserts that for every positive integer $t$,
\[
 \operatorname{dens}\{n\ge0:\wt(n+t)\ge\wt(n)\}>\frac12.
\]
Thus the Tu--Deng problem is a stronger finite statement than the corresponding
sum-of-digits density conjecture.  In particular, Cusick's conjecture is a
corollary of the main result of the present paper (Theorem~\ref{thm:TD}).

Cheng later proved Cusick's conjecture by a cyclic deletion and first-exit argument
\cite{Cheng}.  His proof uses an exact deconvolution, finite stopped random
walks on principal subsequence ideals, and a marked-deletion count that forces
a one-sided median inequality.  Cheng's theorem does not by itself prove the
stronger Tu--Deng conjecture, but its method is the main conceptual stimulus
for the present proof: the normalized coefficient inequalities below (see 
\eqref{normu} and \eqref{normv})  are a cyclic algebraic analogue of the marked-deletion 
inequalities in the first-exit setting.

The proof proceeds as follows.  Section~\ref{sec:prefix} uses binary linear words and
various determinants defined in terms of the words and two variables $u$ and $v$.  Section~\ref{sec:setup}
introduces the cyclic polynomial $H_W$ and states the normalized coefficient
inequalities to be proved.  Sections~\ref{sec:local} and~\ref{sec:differential}
establish the local quotient identities and the positive differential
decomposition.  Section~\ref{sec:transfer} proves the pair-count equivalence and
states the cyclic transfer identity, which is proved in
Section~\ref{sec:proofA}.  Section~\ref{sec:deduction} proves the median
formula and the top-boundary formula, and then deduces the conjecture.  

\section{Binary words and some determinants}\label{sec:prefix}
For variables $u$ and $v$, define matrices
\[
 C_0=\begin{bmatrix}1&0\\u&v\end{bmatrix},\qquad
 C_1=\begin{bmatrix}u&v\\0&1\end{bmatrix}.
\]
We shall use linear binary words.  Let $V=\eta_0\eta_1\cdots\eta_{n-1}$ with each $\eta_i$ in $\{0,1\}$ be such
a word and define the \emph{prefix products} $P_i$ of $V$ by
\[
 P_0=I,\qquad P_m=C_{\eta_0}\cdots C_{\eta_{m-1}}
 \quad(1\le m\le n),\qquad N=P_n.
\]
Define two column vectors \( e_0=\begin{bmatrix}1\\0\end{bmatrix},~e_1=\begin{bmatrix}0\\1\end{bmatrix} \)
and define some columns by
\begin{equation}\label{AB}
 A_V=\sum_{\eta_m=1}P_me_0,
 \qquad
 B_V=\sum_{\eta_m=0}P_me_1,
 \qquad
 p_V=Ne_0,\quad q_V=Ne_1.
\end{equation}

\begin{theorem}[$\det(I-N)$]\label{prop:factor}
For every finite word $V$,
\begin{equation}\label{factor}
 \tr N-\det N=1+(u+v-1)\det(A_V,B_V).
\end{equation}
Consequently there is exactly one polynomial $H_V\in\Z[u,v]$ with
$\tr N-\det N=1+(u+v-1)H_V$, namely
\begin{equation}\label{factor2}
 H_V=\det(A_V,B_V);
\end{equation}
uniqueness holds because $\Z[u,v]$ is an integral domain and $u+v-1\ne0$.
\end{theorem}
\begin{proof} Define 
\( \theta_1=\begin{bmatrix}1-u & -v\end{bmatrix},
~
\theta_0=\begin{bmatrix}-u & 1-v\end{bmatrix}. \)
Since
\[
 I-C_1=e_0\theta_1,\qquad I-C_0=e_1\theta_0,
\]
repeating the product gives
\begin{equation}\label{prod}
 I-N=A_V\theta_1+B_V\theta_0 = \begin{bmatrix}A_V&B_V\end{bmatrix} \begin{bmatrix}1-u&-v\\-u&1-v\end{bmatrix}.
\end{equation}
Taking determinants in \eqref{prod} yields
\[
 \det(I-N)=\det(A_V,B_V)
 \det\begin{bmatrix}1-u&-v\\-u&1-v\end{bmatrix}
 =-(u+v-1)\det(A_V,B_V)
\]
while, for any $2\times2$ matrix $N$,
\[
 \det(I-N)=1-\tr N+\det N .
\]
Comparing the two expressions gives \eqref{factor}.  If $H_V$ satisfies
$\tr N-\det N=1+(u+v-1)H_V$, then $(u+v-1)\bigl(H_V-\det(A_V,B_V)\bigr)=0$, so
cancelling the nonzero polynomial $u+v-1$ in the integral domain $\Z[u,v]$
proves \eqref{factor2}.
\end{proof}

Define six determinants (using $p = p_V, q = q_V, A=A_V, B=B_V$)
\begin{align}\label{sixstates}
 H&=\det(A,B), & \alpha&=\det(A,p), & \beta&=\det(A,q),\nonumber\\
 \gamma&=\det(p,B), & \delta&=\det(q,B), & d&=\det(p,q)=\det N.
\end{align}

\begin{theorem}[Nonnegative six-state recurrence]\label{prop:states}
For the empty word, $A_V$ and $B_V$ are empty sums and hence $0$, while
$N=I$, $p=e_0$ and $q=e_1$; consequently
\[
 H=\alpha=\beta=\gamma=\delta=0,\qquad d=\det(e_0,e_1)=1.
\]
Appending a $1$ gives
\begin{align}
 H'&=H+\gamma,& \alpha'&=u\alpha,&
 \beta'&=\beta+v\alpha+d,\nonumber\\
 \gamma'&=u\gamma,& \delta'&=\delta+v\gamma,& d'&=ud.\label{append1}
\end{align}
Appending a $0$ gives
\begin{align}
 H'&=H+\beta,& \alpha'&=\alpha+u\beta,&
 \beta'&=v\beta,\nonumber\\
 \gamma'&=\gamma+d+u\delta,& \delta'&=v\delta,& d'&=vd.\label{append0}
\end{align}
Consequently all six states have nonnegative integer coefficients of $u$ and $v$ for every word $V$.
\end{theorem}

\begin{proof}
If a $1$ is appended, then
\[
 A'=A+p,\qquad B'=B,\qquad p'=up,\qquad q'=vp+q.
\]
Substitution into the determinants in \eqref{sixstates} gives \eqref{append1}.  If a $0$ is appended, then
\[
 A'=A,\qquad B'=B+q,\qquad p'=p+uq,\qquad q'=vq,
\]
which gives \eqref{append0}.  Every right side is obtained from earlier states using addition and multiplication by $u$ or $v$.
\end{proof}

In particular,
\begin{equation}\label{pos}
 \gamma_V,\delta_V\in\N[u,v],
\end{equation}
where $\N[u,v]$ is the set of all polynomials in $u, v$ with coefficients in $\N =$ the nonnegative integers.

\section{Cyclic polynomial setup and the normalized inequalities}\label{sec:setup}
For $k>1$ define the binary word, regarded as a cyclic word (so the final digit is adjacent to the first digit),
\[
 W=\varepsilon_0\varepsilon_1\cdots\varepsilon_{k-1},
 \qquad \varepsilon_i\in\{0,1\}.
\]
Define $r=$ number of $1$'s in $W$, $z=k-r=$ number of $0$'s in $W$ and
 $M_W=M_W[u,v]=C_{\varepsilon_0}C_{\varepsilon_1}\cdots C_{\varepsilon_{k-1}}.$

Define
\begin{equation}\label{Qdef}
 Q_W(u,v)=\tr M_W(u,v)-u^rv^z.
\end{equation}
When $u+v=1$, every row of $C_0$ and of $C_1$ sums to $1$, so $1$ is an eigenvalue of $M_W$; since $\det C_1=u$ and $\det C_0=v$ give $\det M_W=u^rv^z$, the other eigenvalue is $u^rv^z$, and therefore $Q_W=1$ on the line $u+v=1$.  (No nonnegativity of the entries is used here, and indeed the entries are not all nonnegative on the whole line $u+v=1$; only the row sums matter.)  Hence there is a unique polynomial $H_W\in\Z[u,v]$ (cyclic by properties of the trace function)  satisfying
\begin{equation}\label{Hdef}
 Q_W(u,v)=1+(u+v-1)H_W(u,v),
 \qquad H_W(u,v)=\sum_{i,j\ge0}h_{i,j}(W)u^iv^j.
\end{equation}

The nonnegative six-state recurrence in Theorem~\ref{prop:states} implies
\begin{equation}\label{Hpositive}
 h_{i,j}(W)\in\N.
\end{equation}
Indeed, since $\det C_1=u$ and $\det C_0=v$, we have $\det M_W=u^rv^z$, so that
$Q_W=\tr N-\det N$ with $N=M_W$.  Reading $W$ as a \emph{linear} word, $N$ is
exactly the final matrix of that word, so Theorem~\ref{prop:factor} gives
$H_W=\det(A_W,B_W)$, and Theorem~\ref{prop:states} then gives
$H_W\in\N[u,v]$, which is \eqref{Hpositive}.  (Thus the cyclic quantity $H_W$
defined by \eqref{Hdef} coincides with the linear-word quantity of
Theorem~\ref{prop:factor}.)
The important coefficient inequalities to be proved (in Theorem \ref{thm:normal} below) are
\begin{align} 
 i h_{i,j}(W)&\le(i+j)h_{i-1,j}(W)\qquad(i\ge1,\ j\ge0),\label{normu}\\
 j h_{i,j}(W)&\le(i+j)h_{i,j-1}(W)\qquad(i\ge0,\ j\ge1).\label{normv}
\end{align}
Only the diagonal case (the case $i=j=m$ in Theorem \ref{thm:normal}; see \eqref{weak}) 
\begin{equation}\label{diag}
 h_{m,m}(W)\le2h_{m-1,m}(W)\qquad(m\ge1)
\end{equation}
is needed for Tu--Deng.  
Another way to look at this is to say that the normalized coefficient
\[
 h_{i,j}/\binom{i+j}{i}
\]
must not increase when a one or, symmetrically, a zero, is added to it.

\section{Two local quotient identities}\label{sec:local}
Define
\[
 s=u+v-1,\qquad \rho=\begin{bmatrix}u-1&v\end{bmatrix}.
\]

\begin{lemma}[Local quotient identities]\label{lem:local}
For every finite word $V$ with final matrix $N$ we have
\begin{align}
 \rho Ne_0-(u-1)\det N&=s\gamma_V,\label{local1}\\
 \rho Ne_1-v\det N&=s\delta_V.\label{local0}
\end{align}
\end{lemma}

\begin{proof}
Write
\[
 p=Ne_0=\begin{bmatrix}a\\c\end{bmatrix},\qquad q=Ne_1=\begin{bmatrix}b\\d\end{bmatrix},
 \qquad \Delta=\det N = \det (p,q),
\]
and abbreviate $A=A_V$, $B=B_V$.  The two columns obtained from \eqref{prod} are
\begin{align}
 e_0-p&=(1-u)A-uB,\label{col0}\\
 e_1-q&=-vA+(1-v)B.\label{col1}
\end{align}
Let $X=\det(p,A)$ and $\gamma= \gamma_V=\det(p,B)$. Inserting $p$ as a first column in the three terms in \eqref{col0} gives
\begin{equation} \label{eq:pcol0}
\det(p,e_0-p)=\det(p,e_0)=(1-u)\det(p,A)-u\det(p,B)=(1-u)X-u\gamma.
\end{equation}
Since $\det(p,e_0)=\det\begin{bmatrix}a&1\\c&0\end{bmatrix}= -c$, \eqref{eq:pcol0} implies
\begin{equation} \label{eq:cvalue}
c= -(1-u)X+u\gamma.
\end{equation}
Inserting $p$ as a first column in the three terms in \eqref{col1} gives
\begin{equation} \label{eq:pcol1}
\det(p,e_1-q) =-v\det(p,A)+(1-v)\det(p,B) = -vX + (1-v)\gamma.
\end{equation}
Since $\det(p,e_1-q)=\det\begin{bmatrix}a&0\\c&1\end{bmatrix} - \det(p,q)$, \eqref{eq:pcol1} implies
\begin{equation} \label{eq:avalue}
a-\Delta= -vX+(1-v)\gamma.
\end{equation}

It follows from \eqref{eq:cvalue} and \eqref{eq:avalue} that
\begin{align*}
 \rho p-(u-1)\Delta
 &=-(1-u)(a-\Delta)+vc\\
 &=-(1-u)\{-vX+(1-v)\gamma\}
   +v\{-(1-u)X+u\gamma\}\\
 &=(u+v-1)\gamma.
\end{align*}
This proves \eqref{local1}.

An analogous argument proves \eqref{local0}. We begin by defining $Y=\det(q,A)$ and $\delta= \delta_V= \det(q,B)$.  Inserting $q$ as a first column in each of the three terms of \eqref{col1}, and then in each of the three terms of \eqref{col0}, gives
\begin{equation} \label{eq:bvalue}
 b=-vY+(1-v)\delta
 \end{equation}
 and
 \begin{equation} \label{eq:dvalue}
 d-\Delta=-(1-u)Y+u\delta,
\end{equation}
by computations analogous to those used in proving \eqref{eq:avalue} and \eqref{eq:cvalue}, respectively.
Now it follows from \eqref{eq:bvalue} and \eqref{eq:dvalue} that
\begin{align*}
 \rho q-v\Delta
 &=-(1-u)b+v(d-\Delta)\\
 &=-(1-u)\{-vY+(1-v)\delta\}
   +v\{-(1-u)Y+u\delta\}\\
 &=(u+v-1)\delta,
\end{align*}
which proves \eqref{local0}.
\end{proof}

\section{The positive differential decomposition}\label{sec:differential}
This part of the proof extensively uses the polynomial $H_W(u,v)$ defined in \eqref{Hdef},
often abbreviated to $H_W.$ Define two differentiation operators by

\begin{equation}\label{Du}
 E=u\frac{\partial}{\partial u}+v\frac{\partial}{\partial v},
 \qquad
 D_u=E-\frac{\partial}{\partial u}
     =(u-1)\frac{\partial}{\partial u}
       +v\frac{\partial}{\partial v}.
\end{equation}
For a position $i$ of the cyclic word $W$, define a word
\begin{equation}\label{Vi}
 V(i)=\varepsilon_{i+1}\varepsilon_{i+2}\cdots
     \varepsilon_{k-1}\varepsilon_0\cdots\varepsilon_{i-1}.
\end{equation}
Thus $V(i)$ is obtained by deleting position $i$ and beginning immediately after it.

\begin{theorem}[Positive differential decomposition]\label{thm:main}
For every cyclic word $W$,
\begin{equation}\label{mainidentity}
 \boxed{
 H_W+D_uH_W
 =\sum_{\varepsilon_i=1}\gamma_{V(i)}
  +\sum_{\varepsilon_i=0}\delta_{V(i)}.}
\end{equation}
In particular, $H_W+D_uH_W$ has nonnegative integer coefficients.
\end{theorem}

\begin{proof}
The operator $D_u$ is a derivation, and direct differentiation gives
\[
 D_uC_1=e_0\rho = \begin{bmatrix}u-1&v\\0&0\end{bmatrix},~ D_uC_0=e_1\rho=\begin{bmatrix}0&0\\u-1&v\end{bmatrix}.
\]
Let
\[
 P_i=C_{\varepsilon_0}\cdots C_{\varepsilon_{i-1}},
 \qquad
 S_i=C_{\varepsilon_{i+1}}\cdots C_{\varepsilon_{k-1}}.
\]
By the product rule and cyclic invariance of trace, matrix $M_W$ satisfies
\begin{equation}\label{Dtrace}
 D_u\tr M_W
 =\sum_{\varepsilon_i=1}\rho S_iP_ie_0
  +\sum_{\varepsilon_i=0}\rho S_iP_ie_1.
\end{equation}
The matrix $S_iP_i$ is exactly the matrix product associated with the linear word $V(i)$ in \eqref{Vi}.

Furthermore, we shall prove
\begin{equation}\label{Ddet}
 D_u(u^rv^z)
 =\sum_{\varepsilon_i=1}(u-1)\det(S_iP_i)  + \sum_{\varepsilon_i=0}v\det(S_iP_i).
\end{equation}
The matrix $S_iP_i$ contains all of the factors $C_{\varepsilon_j}$ except the one with $j=i$. If $\varepsilon_i=1$, then since $\det C_1 = u$ and $\det C_0 = v$ we have $\det(S_iP_i)= u^{r-1}v^z$ so 
\begin{equation} \label{case1}
 \sum_{\varepsilon_i=1}(u-1)\det(S_iP_i) =  \sum_{\varepsilon_i=1}(u-1)u^{r-1}v^z = r(u-1)u^{r-1}v^z.
\end{equation}
 If $\varepsilon_i=0$, then $\det(S_iP_i)= u^rv^{z-1}$ so
 \begin{equation} \label{case0}
 \sum_{\varepsilon_i=0}v\det(S_iP_i) =  \sum_{\varepsilon_i=0}vu^rv^{z-1} = zvu^rv^{z-1}.
\end{equation}
Combining \eqref{case1} and \eqref{case0} gives
\begin{align*}
&\sum_{\varepsilon_i=1}(u-1)\det(S_iP_i)
+
\sum_{\varepsilon_i=0}v\det(S_iP_i)
\\[4pt]
&= r(u-1)u^{r-1}v^z+zvu^rv^{z-1}=D_u(u^rv^z),
\end{align*}
which is \eqref{Ddet}.

Subtracting \eqref{Ddet} from \eqref{Dtrace} and applying Lemma~\ref{lem:local} to each $V(i)$ (see \eqref{Qdef}) yields
\begin{equation}\label{DQ}
 D_uQ_W
 =s\left(
 \sum_{\varepsilon_i=1}\gamma_{V(i)}
 +\sum_{\varepsilon_i=0}\delta_{V(i)}\right)
\end{equation}
because, when $N=S_iP_i$, if $\varepsilon_i=1$, \eqref{local1} applied to $V(i)$ gives
$\rho S_iP_i e_0 -(u-1)\det(S_iP_i) = s\gamma_{V(i)}$; and if $\varepsilon_i=0$, \eqref{local0} applied to $V(i)$ gives
$\rho S_iP_i e_1 -v\det(S_iP_i) =s \delta_{V(i)}$.

On the other hand, $Q_W=1+sH_W$ and $D_us=s$, so
\[
 D_uQ_W=s(H_W+D_uH_W).
\]
Cancelling $s$ in the integral domain $\Z[u,v]$ proves \eqref{mainidentity}. The nonnegativity of the coefficients follows from \eqref{pos}.
\end{proof}

\begin{theorem}[Normalized cyclic deletion]\label{thm:normal}
For every nonconstant cyclic word $W$, inequality \eqref{normu} holds for $i\ge1$, $j\ge0$, and inequality \eqref{normv} holds for $i\ge0$, $j\ge1$.
\end{theorem}

\begin{proof}
The coefficient of $u^{i-1}v^j$ in
\[
 H_W+D_uH_W
 =H_W+(u-1)\partial_uH_W+v\partial_vH_W
\]
is
\[
 (i+j)h_{i-1,j}(W)-i h_{i,j}(W).
\]
Theorem~\ref{thm:main} says this is nonnegative, proving \eqref{normu}.

For the reflected inequality, let $\overline W$ be the bit complement of $W$, and define
\[
 J=\begin{bmatrix}0&1\\1&0\end{bmatrix}.
\]
A direct calculation gives
\[
 JC_0(v,u)J=C_1(u,v),\qquad
 JC_1(v,u)J=C_0(u,v).
\]
Hence $H_{\overline W}(u,v)=H_W(v,u)$.  Applying \eqref{normu} to $\overline W$ and interchanging $u$ and $v$ proves \eqref{normv}.
\end{proof}

\section{From the normalized inequalities to Tu--Deng}\label{sec:transfer}
\subsection{The pair-count formulation and the cyclic weight-drop count}

Recall $M=2^k-1$, and let $1\le t<M$.  For an integer $q$, write $\langle q\rangle_M$ for the unique
representative of its residue class modulo $M$ in $\{0,1,\ldots,M-1\}$.  For
$0\le n\le M$, define
\[
 n\oplus_k t=\langle n+t\rangle_M.
\]
Thus the \emph{input} interval $\{0,1,\ldots,M\}$ contains the two $k$-bit
representatives $0=00\cdots0$ and $M=11\cdots1$ of the zero residue, whereas
$n\oplus_k t$ is always taken in the canonical interval $\{0,\ldots,M-1\}$.
Define
\[
 D_{t,k}=\{0\le n\le M:\wt(n\oplus_k t)<\wt(n)\}.
\]

The usual pair set in the Tu--Deng conjecture is
\[
 S_{t,k}=\bigl\{(a,b)\in\{0,1,\ldots,M-1\}^2:
 a+b\equiv t\pmod M,\ \wt(a)+\wt(b)<k\bigr\}.
\]
The conjecture is $|S_{t,k}|\le 2^{k-1}$.  We now prove that this is exactly the
same count as $D_{t,k}$.

\begin{theorem}[Pair-count/cyclic-count equivalence]\label{paircount}
For $k\ge2$ and $1\le t<2^k-1$,
\[
 |S_{t,k}|=|D_{t,k}|.
\]
Consequently the original pair-count formulation of Tu--Deng is equivalent to
\begin{equation}\label{TDtarget}
 |D_{t,k}|\le2^{k-1}.
\end{equation}
\end{theorem}

\begin{proof}
Write the $k$-bit expansion of $a$ as
\[
 a=\sum_{i=0}^{k-1}a_i2^i,\qquad a_i\in\{0,1\}.
\]
Since $M=\sum_{i=0}^{k-1}2^i$, subtraction gives
\[
 M-a=\sum_{i=0}^{k-1}(1-a_i)2^i,
\]
and hence
\begin{equation} \label{eq:compl}
 \wt(M-a)=k-\wt(a)\qquad(0\le a\le M).
\end{equation}

For every $a\in\{0,\ldots,M-1\}$ there is exactly one
$b\in\{0,\ldots,M-1\}$ satisfying $a+b\equiv t\pmod M$, namely
$b=\langle t-a\rangle_M$.  Define
\[
 \Phi(a,b)=M-a.
\]
As $a$ runs through $\{0,\ldots,M-1\}$, the integer
$n=M-a$ runs bijectively through $\{1,\ldots,M\}$.  Moreover,
\[
 b=\langle t-a\rangle_M
  =\langle t+(M-a)\rangle_M
  =n\oplus_k t.
\]
Using the complement identity above, the defining inequality for $S_{t,k}$
becomes
\begin{align*}
 \wt(a)+\wt(b)<k
 &\iff k-\wt(n)+\wt(n\oplus_k t)<k\\
 &\iff \wt(n\oplus_k t)<\wt(n).
\end{align*}
Thus $\Phi$ restricts to a bijection from $S_{t,k}$ onto
$D_{t,k}\cap\{1,\ldots,M\}$.  Finally, $0\notin D_{t,k}$ because
$0\oplus_k t=t$ and $\wt(t)>0=\wt(0)$.  Therefore
$D_{t,k}\cap\{1,\ldots,M\}=D_{t,k}$, proving the equality of counts and hence
the equivalence of the two bounds.
\end{proof}

\subsection{The cyclic transfer identity}

Define $W$ in terms of the bits of $t$ as follows:
\[
 t=\sum_{i=0}^{k-1}\varepsilon_i2^i,
 \qquad W=\varepsilon_0\varepsilon_1\cdots\varepsilon_{k-1},
 \qquad r=\sum_{i=0}^{k-1}\varepsilon_i=\wt(t).
\]
Define
\[
 B_0(x)=C_0(x/2,x^{-1}/2)
 =\begin{bmatrix}1&0\\ x/2&x^{-1}/2\end{bmatrix},
 \qquad
 B_1(x)=C_1(x/2,x^{-1}/2)
 =\begin{bmatrix}x/2&x^{-1}/2\\0&1\end{bmatrix}
\]
and define
\[
  B(t)=\prod_{i=0}^{k-1}B_{\varepsilon_i}(x).
\]
It will also be convenient to have the non-modular addition operation
$\oplus_k'$, which is related to $\oplus_k$: for $0\le n,t\le M$, define
\[
  n\oplus_k' t:=\begin{cases} n+t,& n+t\le M,\\ n+t-M,& n+t>M. \end{cases}
\]

\begin{theorem}[Cyclic transfer identity]\label{cyclictransfer}
For every $k\ge1$ and every $t$ with $0\le t<M$ 
\begin{equation}\label{eq:A}
 \tr\prod_{i=0}^{k-1}B_{\varepsilon_i}(x)
 =2^{-k}\left(
 x^r+\sum_{n=0}^{2^k-1}
 x^{\wt(n\oplus_k t)-\wt(n)}\right).
\end{equation}
Consequently the total coefficient of the negative powers on the left-hand
side is $|D_{t,k}|/2^k$.
\end{theorem}

\section{Proof of the cyclic transfer identity}\label{sec:proofA}
\subsection{The matrix entries as sums over carries}

Put $A_\varepsilon:=2B_\varepsilon$, that is
\[
  A_0=\begin{pmatrix}2&0\\ x&x^{-1}\end{pmatrix},\qquad
  A_1=\begin{pmatrix}x&x^{-1}\\ 0&2\end{pmatrix}.
\]
Then $2^{k}\tr B(t)=\tr\prod_{i=0}^{k-1}A_{\varepsilon_i}$, so \eqref{eq:A} 
is a statement about $\tr\prod_i A_{\varepsilon_i}$. We index rows and columns of
$A_\varepsilon$ by $\{0,1\}$, thought of as \emph{carry} values.

\begin{lemma}\label{lem:entries}
For all $\varepsilon,c,c'\in\{0,1\}$,
\[
  (A_\varepsilon)_{c,c'}\;=\;\sum_{\substack{(n,s)\in\{0,1\}^2\\ n+\varepsilon+c=s+2c'}}x^{\,s-n}.
\]
\end{lemma}

\begin{proof}
Both sides are checked in the eight possible cases. Since $n,s,\varepsilon,c,c'\in\{0,1\}$,
the equation $n+\varepsilon+c=s+2c'$ determines $s$ and $c'$ from $(n,\varepsilon,c)$, namely
$s=(n+\varepsilon+c)\bmod 2$ and $c'=\lfloor (n+\varepsilon+c)/2\rfloor$; so for fixed
$(\varepsilon,c,c')$ we simply collect those $n\in\{0,1\}$ that produce this $c'$.

\smallskip
\noindent$\varepsilon=0$:
\begin{center}
\begin{tabular}{lll}
\toprule
$(c,c')$ & admissible $(n,s)$ & sum\\
\midrule
$(0,0)$ & $(0,0),(1,1)$ & $x^{0}+x^{0}=2$\\
$(0,1)$ & none & $0$\\
$(1,0)$ & $(0,1)$ & $x^{1-0}=x$\\
$(1,1)$ & $(1,0)$ & $x^{0-1}=x^{-1}$\\
\bottomrule
\end{tabular}
\end{center}
This is $A_0$.

\smallskip
\noindent$\varepsilon=1$:
\begin{center}
\begin{tabular}{lll}
\toprule
$(c,c')$ & admissible $(n,s)$ & sum\\
\midrule
$(0,0)$ & $(0,1)$ & $x^{1-0}=x$\\
$(0,1)$ & $(1,0)$ & $x^{0-1}=x^{-1}$\\
$(1,0)$ & none & $0$\\
$(1,1)$ & $(0,0),(1,1)$ & $x^{0}+x^{0}=2$\\
\bottomrule
\end{tabular}
\end{center}
This is $A_1$. 
\end{proof}

\subsection{The trace as a sum over carry-consistent triples}

Indices $i$ are read modulo $k$ from now on.

\begin{definition}
A triple $(n,c,s)$ with $n,c,s\in\{0,1\}^{k}$ is \emph{carry-consistent} (for $t$) if
\begin{equation}\label{eq:cons}
  n_i+\varepsilon_i+c_i=s_i+2c_{i+1}\qquad\text{for all i mod k}.
\end{equation}
\end{definition}

\begin{lemma}\label{lem:trace}
\[
  \tr\prod_{i=0}^{k-1}A_{\varepsilon_i}
  \;=\;\sum_{(n,c,s)\ \mathrm{carry\text{-}consistent}}x^{\,\wt(s)-\wt(n)} .
\]
\end{lemma}

\begin{proof}
For any $2\times2$ matrices, $\tr\prod_{i=0}^{k-1}A_{\varepsilon_i}
=\sum_{c_0,\dots ,c_{k-1}\in\{0,1\}}\prod_{i=0}^{k-1}(A_{\varepsilon_i})_{c_i,c_{i+1}}$ with
$c_k=c_0$. Apply Lemma~\ref{lem:entries} to each factor and expand the product: a term of the
expansion is a choice, for each $i$, of a pair $(n_i,s_i)$ with
$n_i+\varepsilon_i+c_i=s_i+2c_{i+1}$, and its value is
$\prod_i x^{s_i-n_i}=x^{\sum_i (s_i-n_i)}=x^{\wt(s)-\wt(n)}$, because $n,s\in\{0,1\}^k$.
Summing over $c$ and over the choices gives exactly the sum over carry-consistent triples.
\end{proof}

\subsection{What carry consistency means}

\begin{lemma}\label{lem:congr}
If $(n,c,s)$ is carry-consistent then, as integers,
\[
  s=n+t-c_0M .
\]
In particular $s\equiv n+t \pmod M$ and $c_0=0$ or $1$ according to whether $n+t\le M$ or not,
except when $n+t=M$, where both values may occur.
\end{lemma}

\begin{proof}
Multiply \eqref{eq:cons} by $2^{i}$ and sum over $0\le i\le k-1$. With $C:=\sum_i c_i2^{i}$,
\[
  \sum_{i=0}^{k-1}2\,c_{i+1}2^{i}=\sum_{i=0}^{k-1}c_{i+1}2^{i+1}
  =\sum_{j=1}^{k-1}c_j2^{j}+c_02^{k}=(C-c_0)+c_02^{k}=C+c_0M .
\]
Hence $n+t+C=s+C+c_0M$, i.e. $s=n+t-c_0M$. Since $0\le s\le M$ and $0\le n+t\le 2M$, the value
of $c_0$ is determined unless both $n+t$ and $n+t-M$ lie in $[0,M]$, which happens only for
$n+t=M$.
\end{proof}

\begin{lemma}\label{lem:count}
Fix $n\in\{0,\dots ,M\}$. The carry-consistent triples with first component $n$ are in bijection
with the solutions $c_0\in\{0,1\}$ of a fixed-point equation, and:
\begin{enumerate}
\item[(a)] if $n_i+\varepsilon_i=1$ for all $i$, i.e. $n=M-t$, there are exactly two of them,
      namely $c\equiv 0$, which gives $s=M$, and $c\equiv 1$, which gives $s=0$;
\item[(b)] otherwise there is exactly one, and its $s$ equals $n\oplus_k' t$.
\end{enumerate}
\end{lemma}

\begin{proof}
Given $n$ and $c_0$, equation \eqref{eq:cons} determines $s_i$ and $c_{i+1}$ recursively:
\[
  c_{i+1}=f_i(c_i),\qquad f_i(c):=\Big\lfloor\tfrac{n_i+\varepsilon_i+c}{2}\Big\rfloor,
  \qquad s_i=(n_i+\varepsilon_i+c_i)\bmod 2 .
\]
The only remaining constraint is the cyclic closure $c_k=c_0$, i.e. $c_0=F(c_0)$ with
$F:=f_{k-1}\circ\dots\circ f_0$. Now
\[
  f_i=\begin{cases}
  \text{constant }0, & n_i+\varepsilon_i=0,\\
  \text{identity}, & n_i+\varepsilon_i=1,\\
  \text{constant }1, & n_i+\varepsilon_i=2 .
  \end{cases}
\]

(a) If $n_i+\varepsilon_i=1$ for all $i$ — equivalently $n_i=1-\varepsilon_i$ for all $i$,
i.e. $n=M-t$ — then every $f_i$ is the identity, so $F=\mathrm{id}$ and both $c_0=0$ and
$c_0=1$ solve $c_0=F(c_0)$. For $c_0=0$ the recursion gives $c\equiv 0$ and
$s_i=n_i+\varepsilon_i=1$ for all $i$, i.e. $s=M$; for $c_0=1$ it gives $c\equiv1$ and
$s_i=n_i+\varepsilon_i+1-2=0$ for all $i$, i.e. $s=0$. (Both are consistent with
Lemma~\ref{lem:congr}, since here $n+t=M$.)

(b) Otherwise some $f_j$ is constant, hence $F$ is a constant map, say $F\equiv a$; then
$c_0=F(c_0)$ has the unique solution $c_0=a$, so there is exactly one carry-consistent triple.
Its $s$ satisfies $s=n+t-c_0M$ by Lemma~\ref{lem:congr}. Since we are not in case (a), we have
$n+t\neq M$, so: if $n+t<M$ then necessarily $c_0=0$ and $s=n+t=n\oplus_k' t$; if $n+t>M$ then
$c_0=0$ would give $s=n+t>M$, impossible, so $c_0=1$ and $s=n+t-M=n\oplus_k' t$. 
\end{proof}

\subsection{Conclusion of the proof of Theorem~\ref{cyclictransfer} }
Assume $0\le t<M$.  By Lemma~\ref{lem:trace},
\[
  2^{k}\tr B(t)=\tr\prod_{i=0}^{k-1}A_{\varepsilon_i}
  =\sum_{(n,c,s)\ \mathrm{carry\text{-}consistent}}x^{\,\wt(s)-\wt(n)} ,
\]
and the first components $n$ of the carry-consistent triples run over all $k$-bit
strings, that is, over $\{0,1,\ldots,M\}$.  By Lemma~\ref{lem:count}, every such
$n$ contributes exactly one triple, whose $s$ equals $n\oplus_k' t$, except
$n=M-t$, which contributes two triples: the one with $s=M=(M-t)\oplus_k' t$,
which is again the term counted by $\oplus_k'$, and one more with $s=0$.  The
extra term contributes, by \eqref{eq:compl},
\[
  x^{\,\wt(0)-\wt(M-t)}=x^{-(k-\wt(t))}=x^{\,\wt(t)-k}.
\]
Hence
\begin{equation}\label{eq:Bcorrected}
  2^{k}\tr B(t)\;=\;x^{\,\wt(t)-k}\;+\;\sum_{n=0}^{2^{k}-1}x^{\,\wt(n\oplus_k' t)-\wt(n)} .
\end{equation}

It remains to replace $\oplus_k'$ by $\oplus_k$. Since $0\le n+t\le2M$, the two operations differ only when $n+t=M$ or $2M$; in either case $n\oplus_k' t=M$, whereas $n\oplus_k t=0$.  The case $n+t=2M$ forces $n=t=M$, which is excluded here, and $n+t=M$ holds only for $n=M-t$.  Therefore, using \eqref{eq:compl} again,
\begin{align*}
  \sum_{n=0}^{M}x^{\,\wt(n\oplus_k' t)-\wt(n)}
  &=\sum_{n=0}^{M}x^{\,\wt(n\oplus_k t)-\wt(n)}
   +\Bigl(x^{\,k-\wt(M-t)}-x^{\,0-\wt(M-t)}\Bigr)\\
  &=\sum_{n=0}^{M}x^{\,\wt(n\oplus_k t)-\wt(n)}+x^{\wt(t)}-x^{\wt(t)-k}.
\end{align*}
Substituting this into \eqref{eq:Bcorrected} makes the terms
$\pm x^{\wt(t)-k}$ cancel and gives
\[
  2^{k}\tr B(t)=x^{\wt(t)}+\sum_{n=0}^{2^{k}-1}x^{\,\wt(n\oplus_k t)-\wt(n)},
\]
which is \eqref{eq:A}, since $r=\wt(t)$.

For the last assertion, $\wt(n\oplus_kt)-\wt(n)<0$ precisely for the $|D_{t,k}|$
values $n\in D_{t,k}$, and $r\ge0$, so the total coefficient of the negative
powers of $x$ on the left-hand side of \eqref{eq:A} is $|D_{t,k}|/2^k$.

\begin{remark}\label{rem:tM}
For $t=M$ the two additions differ at $n=M-t=0$ \emph{and} at $n=M$ (where $n+t=2M$), so the
same computation gives
\[
  2^{k}\tr B(M)=\sum_{n=0}^{2^k-1}x^{\,\wt(n\oplus_k M)-\wt(n)}+\bigl(x^{k}-x^{-k}+1\bigr),
\]
and $x^{k}-x^{-k}+1$ is not a monomial for any $k\ge1$; so \eqref{eq:A} admits no exponent $r$
when $t=M$. For $k=1$, $t=M=1$ one checks directly $2\tr B(1)=\tr A_1=x+2$, whereas
$x^{r}+\sum_{n=0}^{1}x^{\wt(n\oplus_k 1)-\wt(n)}=x^{r}+1+x^{-1}$.
\end{remark}

\section{Deduction of the conjecture}\label{sec:deduction}

Equation \eqref{Qdef} now gives
\begin{equation}\label{traceQ}
 \tr\prod_iB_{\varepsilon_i}(x)
 =Q_W(x/2,x^{-1}/2)+2^{-k}x^{2r-k}.
\end{equation}

\begin{lemma}[Positivity and normalization of $Q_W$]\label{lem:Qpositive}
For every nonconstant cyclic word $W$,
\[
 Q_W(u,v)\in\N[u,v],\qquad Q_W(1/2,1/2)=1,
 \qquad h_{0,0}(W)=1.
\]
Consequently, after the substitution $u=x/2$, $v=x^{-1}/2$, the Laurent
coefficients of $Q_W$ are nonnegative and have total mass $1$.
\end{lemma}

\begin{proof}
Expand $\tr M_W$ over cyclic state strings
$c_0,c_1,\ldots,c_k\in\{0,1\}$ with $c_k=c_0$.  Every matrix entry of
$C_0$ and $C_1$ is either $0$, $1$, $u$, or $v$, so this expansion has
nonnegative integer coefficients.  The particular cyclic state string
\[
 c_i=1-\varepsilon_i\qquad(0\le i<k)
\]
is always admissible.  Indeed, the entry
$(C_{\varepsilon_i})_{c_i,c_{i+1}}$ is $u$ when
$\varepsilon_{i+1}=1$ and is $v$ when $\varepsilon_{i+1}=0$, with indices
read cyclically.  Its product is therefore $u^rv^z$.  Thus the trace contains
at least one copy of $u^rv^z$, and subtracting the single copy in
\eqref{Qdef} leaves a polynomial with nonnegative integer coefficients.

We have $Q_W(1/2,1/2)=1$ from \eqref{Hdef}. Finally, at $u=v=0$ the two matrices are the complementary projections
\[
 C_0(0,0)=\begin{bmatrix}1&0\\0&0\end{bmatrix},\qquad
 C_1(0,0)=\begin{bmatrix}0&0\\0&1\end{bmatrix}.
\]
Because $W$ is nonconstant, a cyclic rotation of its product contains one of
the adjacent factors
\[
 C_0(0,0)C_1(0,0),\qquad C_1(0,0)C_0(0,0),
\]
both of which are zero.
Thus $Q_W(0,0)=0$.  Substituting $u=v=0$ in \eqref{Hdef} gives $h_{0,0}(W)=1$.
\end{proof}

For a nonconstant cyclic word $W$, write
\begin{equation}\label{sigmadef}
 \sigma^+(W):=\sum_{e\ge0}[x^e]\,Q_W(x/2,x^{-1}/2),
 \qquad
 \sigma^-(W):=\sum_{e<0}[x^e]\,Q_W(x/2,x^{-1}/2)
\end{equation}
for the total coefficient of the nonnegative, respectively negative, powers of
$x$ after the substitution $u=x/2$, $v=x^{-1}/2$.  By
Lemma~\ref{lem:Qpositive},
\begin{equation}\label{sigmasum}
 \sigma^+(W)+\sigma^-(W)=Q_W(1/2,1/2)=1 ,
\end{equation}
so a lower bound for $\sigma^+(W)$ is the same thing as an upper bound for
 $\sigma^-(W)$.
 
\begin{lemma}[Exact median identity]\label{lem:median}
For every nonconstant cyclic word $W$,
\begin{equation}\label{median}
 \sigma^+(W)
 =\frac12+\sum_{m\ge1}2^{-2m-1}\Delta_m^-(W),
 \qquad
 \Delta_m^-(W)=2h_{m-1,m}(W)-h_{m,m}(W).
\end{equation}
\end{lemma}

\begin{proof}
After the substitution $u=x/2$, $v=x^{-1}/2$, a monomial $h_{i,j}u^iv^j$ of
$H_W$ contributes to $Q_W=1+(u+v-1)H_W$ through the three terms $uH_W$, $vH_W$
and $-H_W$, that is, at the Laurent exponents $i-j+1$, $i-j-1$ and $i-j$, with
the respective weights
\[
 2^{-(i+j)-1}h_{i,j},\qquad
 2^{-(i+j)-1}h_{i,j},\qquad
 -2^{-(i+j)}h_{i,j}.
\]
Four cases occur.  If $i-j\ge1$, all three exponents are nonnegative and the
three contributions to $\sigma^+(W)$ cancel.  If $i=j$, the exponent $i-j-1=-1$
is negative, so the net contribution is
$2^{-2i-1}h_{i,i}-2^{-2i}h_{i,i}=-2^{-2i-1}h_{i,i}$.  If $j=i+1$, only the
exponent $i-j+1=0$ is nonnegative, so the net contribution is
$2^{-(2i+1)-1}h_{i,i+1}=2^{-2i-2}h_{i,i+1}$.  If $j\ge i+2$, all three
exponents are negative and the contribution is zero.  Finally the constant $1$
in $Q_W=1+(u+v-1)H_W$ contributes $1$, which the $i=j=0$ term
$-2^{-1}h_{0,0}=-1/2$ reduces to $1/2$.  Summing, and reindexing the
adjacent-diagonal terms by $m=i+1$, gives \eqref{median}.
\end{proof}

Equation \eqref{diag} gives
\begin{equation}\label{weak}
 \Delta_m^-(W)\ge0\qquad(m\ge1).
\end{equation}

\begin{theorem}[Top-boundary formula]\label{prop:topboundary}
After a cyclic rotation, write
\[
 W=1^{a_1}0^{b_1}\cdots1^{a_p}0^{b_p},
 \qquad a_\nu,b_\nu\ge1.
\]
Then
\begin{equation}\label{topboundary}
 h_{r-1,j}(W)
 =[y^j]\prod_{\nu=1}^p(1+y+\cdots+y^{b_\nu})
 \quad(0\le j<z),
 \qquad h_{r,j}(W)=0.
\end{equation}
\end{theorem}

\begin{proof}
Cyclic invariance of the trace permits the displayed run decomposition.  Define
\[
 F_a=C_1^a=
 \begin{bmatrix}
  u^a&v(1+u+\cdots+u^{a-1})\\0&1
 \end{bmatrix},
 \qquad
 G_b=C_0^b=
 \begin{bmatrix}
  1&0\\u(1+v+\cdots+v^{b-1})&v^b
 \end{bmatrix}.
\]
Expand
$\tr(F_{a_1}G_{b_1}\cdots F_{a_p}G_{b_p})$ over states
$s_\nu,t_\nu\in\{0,1\}$, where $s_\nu$ is the state entering the
$\nu$th one-run, $t_\nu$ is the state entering the following zero-run, and
$s_{p+1}=s_1$.  A nonzero contribution has
$s_\nu\le t_\nu$ in $F_{a_\nu}$ and
$t_\nu\ge s_{\nu+1}$ in $G_{b_\nu}$.

For a monomial chosen from the entry
$(F_{a_\nu})_{s_\nu,t_\nu}$, its $u$-degree is at most
\[
 a_\nu(1-s_\nu)-(t_\nu-s_\nu).
\]
This is an equality for the sole entry $0\to0$; for $0\to1$ it is the
largest exponent $a_\nu-1$ in
$v(1+u+\cdots+u^{a_\nu-1})$; and for $1\to1$ both sides are zero.
A monomial from $(G_{b_\nu})_{t_\nu,s_{\nu+1}}$ has $u$-degree at most
$t_\nu-s_{\nu+1}$, with equality whenever a $u$ occurs.  Summing these
bounds around the cycle gives
\begin{align*}
 \deg_u(\text{chosen monomial})
 &\le \sum_{\nu=1}^p
 \bigl(a_\nu(1-s_\nu)+s_\nu-s_{\nu+1}\bigr)\\
 &=r-\sum_{\nu=1}^p a_\nu s_\nu\le r.
\end{align*}
Thus $\deg_u\tr M_W\le r$, and hence $\deg_uQ_W\le r$.

Equality in the $u$-degree can hold only when every $s_\nu=0$ and each local
bound is attained.  For a fixed run pair there are then two possibilities.
If $t_\nu=0$, the top-$u$ contribution is simply $u^{a_\nu}$.  If
$t_\nu=1$, the top-$u$ term from the one-run is
$u^{a_\nu-1}v$, while the zero-run contributes
$u(1+v+\cdots+v^{b_\nu-1})$.  Their product is
\[
 u^{a_\nu}(v+v^2+\cdots+v^{b_\nu}).
\]
The choices for the $p$ run pairs are independent because each begins and ends
in state $0$.  Therefore
\[
 [u^r]\tr M_W
 =\prod_{\nu=1}^p(1+v+\cdots+v^{b_\nu}).
\]
Since $Q_W=\tr M_W-u^rv^z$, it follows that
\[
 [u^r]Q_W
 =\prod_{\nu=1}^p(1+v+\cdots+v^{b_\nu})-v^z.
\]

It remains to pass from $Q_W$ to $H_W$.  If $d=\deg_uH_W$, then the term
$uH_W$ in $Q_W=1+(u+v-1)H_W$ has $u$-degree $d+1$, whereas
$(v-1)H_W$ has $u$-degree at most $d$.  Hence $\deg_uQ_W\le r$ forces
$d\le r-1$.  Thus $h_{r,j}(W)=0$ for every $j$, and
\[
 [u^rv^j]Q_W=h_{r-1,j}(W).
\]
For $0\le j<z$, the subtracted term $v^z$ does not contribute, giving exactly
\eqref{topboundary}.
\end{proof}

Suppose $2r<k$, so that $r<z$.  Each factor $1+y+\cdots+y^{b_\nu}$ of the
product in \eqref{topboundary} has all its coefficients in the degrees
$0,1,\ldots,b_\nu$ equal to $1$; hence the product
$\prod_{\nu=1}^p(1+y+\cdots+y^{b_\nu})$ has \emph{every} coefficient in the
degrees $0,1,\ldots,z$ at least $1$ (indeed $\sum_\nu b_\nu=z$, and a coefficient
of $y^j$ with $0\le j\le z$ is the number of ways of writing
$j=j_1+\cdots+j_p$ with $0\le j_\nu\le b_\nu$, which is at least one).  Since
$r<z$, the degree $j=r$ lies in the range $0\le j<z$ of validity of
\eqref{topboundary}, and therefore
\begin{equation}\label{strictdefect}
 h_{r-1,r}(W)\ge1,\qquad h_{r,r}(W)=0,
 \qquad \Delta_r^-(W)=2h_{r-1,r}(W)-h_{r,r}(W)\ge2.
\end{equation}

\begin{theorem}[Tu--Deng conjecture]\label{thm:TD}
For every $k\ge2$ and every $1\le t<2^k-1$,
\[
 |D_{t,k}|\le2^{k-1}.
\]
By Theorem~\ref{paircount}, the Tu--Deng Conjecture~\ref{TD} holds.
\end{theorem}

\begin{proof}
Since $1\le t<2^k-1$, we have $1\le r\le k-1$, so $W$ is nonconstant.
By \eqref{weak} and Lemma~\ref{lem:median},
\[
 \sigma^+(W)\ge\frac12.
\]
By Lemma~\ref{lem:Qpositive}, the Laurent coefficients of
$Q_W(x/2,x^{-1}/2)$ are nonnegative and have total mass $1$.  Thus, by
\eqref{sigmasum}, the negative Laurent mass $\sigma^-(W)$ is at most $1/2$.

If $2r\ge k$, the determinant term $2^{-k}x^{2r-k}$ in \eqref{traceQ} has nonnegative exponent and adds no negative mass.  Equation \eqref{eq:A} then gives \eqref{TDtarget}.

Suppose $2r<k$.  By \eqref{strictdefect}, the $m=r$ summand in \eqref{median} contributes at least
\[
 2^{-2r-1}\Delta_r^-(W)\ge2^{-2r}
\]
beyond $1/2$.  Hence the negative mass $\sigma^-(W)$ of $Q_W$ is at most
\[
 \frac12-2^{-2r}.
\]
The determinant term now has a negative exponent, and so contributes $2^{-k}$
to the negative mass.  Since $2r<k$,
\[
 2^{-k}<2^{-2r}.
\]
Therefore the total negative mass in \eqref{traceQ} remains below $1/2$.  The trace identity \eqref{eq:A} again gives $|D_{t,k}|\le2^{k-1}$.
\end{proof}

\end{document}